\documentclass{elsarticle}%
\usepackage{amsfonts}
\usepackage{amsthm}
\usepackage{amsmath}
\usepackage{amssymb}
\usepackage{bbold}
\usepackage{mathtools}
\usepackage{xcolor}
\usepackage{soul}
\usepackage{mathrsfs}
\usepackage{tabularx}
\usepackage[top=1in, bottom=1in, left=1in, right=1in]{geometry}
\usepackage{graphicx}%
\providecommand{\U}[1]{\protect\rule{.1in}{.1in}}

\newtheorem{theorem}{Theorem}[section]
\newtheorem{lemma}{Lemma}[section]
\newtheorem{corollary}{Corollary}[section]

\theoremstyle{definition}
\newtheorem{definition}{Definition}[section]
\theoremstyle{statement}
\newtheorem{statement}{Statement}[section]
\theoremstyle{remark}

\numberwithin{equation}{section}
\begin{document}
	\begin{frontmatter}
		
		\title{On solvability in the small of higher order elliptic equations in Orlicz-Sobolev spaces}
		
		
		
		\author[]{Javad A. Asadzade}
		\ead{javad.asadzade@emu.edu.tr}

		\address{Department of Mathematics, Eastern Mediterranean University, Mersin 10, 99628, T.R. North Cyprus, Turkey}
		
		

		\begin{abstract}
			\noindent
			
		In this article,  we consider a higher-order elliptic equation with nonsmooth coefficients with respect to Orlicz spaces on the domain $\Omega\subset\mathbb{R}^{n}$. Separable subspace of this space is distinguished in which infinitely differentiable and compactly supported functions are dense; Sobolev spaces generated by these subspaces are determined. We demonstrate the local solvability of the equation in Orlicz-Sobolev spaces under specific restrictions on the coefficients of the equation and the Boyd indices of the Orlicz space. This result strengthens the previously known classical $L_{p}$ analog. 
			
		\end{abstract}
	
	\end{frontmatter}


	\section{INTRODUCTION}
The main issue in the theory of differential operators has always been the solvability of differential equations. Different concepts of solution have been introduced, such as classical solution, strong solution, weak solution, etc., to address these problems. It should be noted that the solvability of a differential equation, in some form or another, depends on the coefficients of the equation, the domain in which the problem is considered, and the space in which the solution is sought. Many mathematicians have written significant monographs on these problems, providing valuable insights (see, e.g. \cite{1},\cite{2},\cite{8}-\cite{12}). Solvability problems for second-order elliptic equations in Hölder, Sobolev, and other classical spaces have been extensively studied in various monographs (see, e.g. \cite{2},\cite{9},\cite{11},\cite{12}). These problems are still relevant and new results are continually being obtained for new spaces ( see, e.g. \cite{39}).

In recent times, there has been a growing interest in exploring non-standard function spaces in the context of various problems in mechanics and mathematical physics. Examples of these spaces include Lebesgue spaces with variable summability index, Morrey spaces(\cite{13},\cite{15},\cite{24},\cite{36},\cite{38}), Orlicz spaces(\cite{6},\cite{7}, \cite{14},\cite{18},\cite{21}-\cite{23},\cite{30},\cite{33},\cite{41}), and grand-Lebesgue spaces(\cite{3},\cite{16}-\cite{19},\cite{28}). There has been significant research and progress in the analysis and approximation theory within these function spaces, with notable advancements in Lebesgue spaces with variable summability index, Morrey spaces, and grand-Lebesgue spaces. Notable breakthroughs have been made in the field of harmonic analysis and approximation within grand Lebesgue spaces(\cite{27}). Furthermore, there has been a emergence of studies on solvability problems related to (partial) differential equations within non-standard function spaces(\cite{39}).

For example:  Bilalov and Sadigova (\cite{16}) presented the main results for solvability in the small of higher order elliptic equations in the grand-Sobolev space. Subsequently, Bilalov, Zeren, Sadigova, and Cetin further strengthened this research by obtaining Schauder-type estimates for higher order elliptic equations in grand Sobolev spaces (\cite{3}). Furthermore, Bilalov and Sadigova (\cite{17}) derived interior Schauder-type estimates for higher order elliptic operators in grand-Sobolev spaces. (For more comprehensive information, see \cite{3}, \cite{15}-\cite{19}, \cite{29}). So, the Orlicz-Sobolev analog of these results is of interest in mathematical research.

Elliptic equations with nonsmooth coefficients have been extensively studied in the field of partial differential equations. These equations arise in various applications, such as fluid dynamics, solid mechanics, and mathematical physics. The behavior of solutions to such equations can be complex and challenging to analyze due to the lack of smoothness in the coefficients.

In this article, we consider a higher-order elliptic equation with nonsmooth coefficients in the context of Orlicz spaces. Orlicz spaces are function spaces that capture the properties of both Lebesgue and Sobolev spaces and are well-suited for the study of equations with nonsmooth coefficients.

The main focus of our study is to investigate the solvability of the equation in Orlicz-Sobolev spaces. These spaces are generated by separable subspaces of Orlicz spaces wherein infinitely differentiable and compactly supported functions are dense. By characterizing the properties of these spaces, we can determine the solvability of the equation in a local sense.

To establish the solvability result, we impose certain conditions on the coefficients of the equation and the Boyd indices of the Orlicz space. These conditions ensure the well-posedness of the problem and guarantee the existence of a local solution in the Orlicz-Sobolev spaces. This result extends and strengthens the previously known classical  $L_{p}$ analog, where the smoothness of the coefficients is assumed.

The study of elliptic equations with nonsmooth coefficients in Orlicz spaces is of significant importance in the field of partial differential equations. It provides a more general framework for understanding and analyzing the behavior of solutions to such equations, allowing for the consideration of a wider class of coefficients. Moreover, the study of Orlicz-Sobolev spaces and their properties contributes to the development of mathematical tools and techniques for the analysis of PDEs in non-smooth settings.

In summary, this article contributes to the understanding of elliptic equations with nonsmooth coefficients in the context of Orlicz spaces. It establishes the local solvability of the equation in Orlicz-Sobolev spaces, extending the previously known results in the classical $L_{p}$ setting. This study opens up new avenues for further research in the field of partial differential equations and provides valuable insights for the analysis of equations with nonsmooth coefficients.

	\section{PRELIMINARIES}	
	We will use the following standart notations. $\mathbb{Z}_{+}$ represents the set of of non-negative integers. $\mathbb{B}_{r}(x_{0})=\lbrace x\in \mathbb{R}^{n}: \vert x-x_{0}\vert<r\rbrace$ will mention the open ball in $\mathbb{R}^{n}$ centred at $x_{0}$, where $\vert x\vert=\sqrt{x^{2}_{1}+\cdots+x^{2}_{n}}, \quad x=(x_{1},\cdots,x_{n}). \quad \Omega_{r}(x_{0})=\Omega\cap \mathbb{B}_{r}(x_{0}), \quad \mathbb{B}_{r}=\mathbb{B}_{r}(0), \quad \Omega_{r}=\Omega_{r}(0).$ $mes(M)$ will stand for the Lebesgue measure of the set $M$, $\partial\Omega$ will be the boundary of the domain $\Omega$, $\bar{\Omega}=\Omega\cup \partial\Omega$,$M_{1}\Delta M_{2}$ will denote the symmetric difference between the sets $M_{1}$ and $M_{2}$, $diam(\Omega)$ will denote for the diameter of the set $\Omega$, $\rho(x,M)$ will be the distance between $X$ and the set $M$, and $\Vert T\Vert_{[X,Y]}$ will mention the norm of the operator $T$, acting boundedly from $X$ to $Y$.
	\subsection{\textbf{Elliptic operator of mth order}}
	Assume $\Omega\subset \mathbb{R}^{n}$ be some bounded domain with the refictable boundary $\partial\Omega$. We will use the notation of (\cite{2}). $\alpha=(\alpha_{1},\cdots,\alpha_{n})$ will be the multiindex with the coordinates $\alpha_{k}\in Z_{+}, \forall k=\overline{1,n}$, $\partial_{i}=\frac{\partial}{x_{i}}$ will mention the differentiation operator, $\partial^{\alpha}=\partial^{\alpha_{1}}_{1}\partial^{\alpha_{2}}_{2}\cdots \partial^{\alpha_{n}}_{n}$. For all $\eta=(\eta_{1},\cdots,\eta_{n})$ we suppose $\eta^{\alpha}=\eta^{\alpha_{1}}_{1}\eta^{\alpha_{2}}_{2}\cdots \eta^{\alpha_{n}}_{n}$. Assume $L$ become an elliptic differential operator of mth order
	\begin{align}
		\mathrm{L} =&\sum_{\vert p\vert \leq m} a_{p}(x)\partial ^{p}, 
	\end{align}
where $p=(p_{1},\cdots,p_{n}), \quad p_{k}\in Z_{+},\quad \forall k=\overline{1,n}, \quad \vert p\vert =\sum_{k=1}^{n} p_{k},\quad a_{p}(\cdot)\in L_{\infty}(\Omega)$ are real functions, i.e. the characteristic form
\begin{align}
	Q(x,\eta)=\sum_{\vert p\vert = m} a_{p}(x)\eta^{p},
\end{align}
is defined at every point $x\in\Omega$. It is known that in this case $m$ is even. Let $m=2m^{\prime}$, and assume without loss of generality that
\begin{align*}
	(-1)^{m^{\prime}}Q(x,\eta)>0, \quad \forall \eta\not =0, \quad \forall x\in \Omega.
\end{align*}
Consider the elliptic operator $L_{0}$:
	\begin{align}
		\mathrm{L} _{0}=&\sum_{\vert p\vert = m} a^{0}_{p}\partial ^{p}, 
	\end{align}
	with the constant coefficients $a^{0}_{p}$.
	In what follows, by solution of the equation $Lu=f$ we mean a strong solution (see, \cite{2}). We will need the following clasical result of \cite{2}.
	\begin{theorem}(\cite{2})
		
		For any mth order elliptic operator $L_{0}$ of the form (2.3) with the constant coefficients, the function $\mathbb{J}(x)$ can be constructed which has the below properties:
		
		1). If n is odd  or n is even and $n>m$, hence
		\begin{align*}
			\mathbb{J}(x)=\frac{\omega(x)}{\vert x\vert ^{n-m}},
		\end{align*}
	where $\omega(x)$ is a positive homogeneous function of degree zero ($\omega(tx)=\omega(x),\quad \forall t>0$)
	
	If $n$ is even and $n\leq m$, hence
	\begin{align*}
		\mathbb{J}(x)=q(x)\log\vert x\vert+\frac{\omega(x)}{\vert x\vert ^{n-m}},
	\end{align*}
	where $q$ is homogeneous polynomial of degree $m-n$.

	2). The function $\mathbb{J}(x)$ satisfies (in a generalized sense) the equation
	\begin{align*}
		L_{0}\mathbb{J}(x)=\delta(x),
	\end{align*}
where, $\delta$ is a Dirac function, so the following equality is true for every infinitely differentiable function $\varphi(\cdot)$ with compact support
\begin{align*}
	\varphi(x)=\int [L_{0}\varphi(y)]\mathbb{J}(x-y)dy=L_{0} \int \varphi(y)\mathbb{J}(x-y)dy.
\end{align*}
	\end{theorem}
Let's examine the elliptic operator (2.1) and define a "tangential operator" for it.
\begin{align}
	L_{x_{0}}=\sum_{\vert p\vert = m} a_{p}(x_{0}) \partial ^{p},
\end{align}
to it at every point $x_{0}\in \Omega$. The term "parametrics" is used to refer to the function $\mathbb{J}_{x_{0}}(\cdot)$, which is the fundamental solution of the equation $L_{x_{0}}\varphi=0$ according to Theorem 2.1. This function is specifically associated with the equation $L\varphi=0$ with a singularity at the point $x_{0}$. Assume
\begin{align*}
	&\mathcal{S}_{x_{0}}\varphi=\Psi(x)=\int \mathbb{J}_{x_{0}}(x-y)\varphi(y)dy,
\end{align*}
and
\begin{align}
	\mathbb{T}_{x_{0}}=\mathcal{S}_{x_{0}}(L_{x_{0}}-L).
\end{align}
The following lemma is crucial in proving the existence of a solution to the equation $Lu=f$.
\begin{lemma}(\cite{2})
	If $\varphi$ has compact support, then 
	\begin{align*}
		\varphi=\mathbb{T}_{x_{0}}\varphi+\mathcal{S}_{x_{0}}L\varphi,
	\end{align*}
and \quad if
\begin{align*}
	\varphi=\mathbb{T}_{x_{0}}\varphi+\mathcal{S}_{x_{0}}f,
\end{align*}
then $L\varphi=f$.
\end{lemma}
	\subsection{\textbf{\textit{Orlicz spaces}}}
	To begin with, this section will provide an overview of Orlicz space and Orlicz-Sobolev space by presenting relevant information.
	
	\begin{definition}(\cite{30})
		Continuous convex function $M (u)$ in $R$ is called an $N$-function or Young function, if it is even and satisfies the conditions
		\begin{align*}
			\lim_{u\to 0}\frac{M(u)}{u}=0,\quad \lim_{u\to \infty}\frac{M(u)}{u}=\infty.
		\end{align*}
	\end{definition}
\begin{definition}(\cite{30})
	Let $M$ be an $N$-function. The function
	\begin{align*}
		N(v)=\max_{u\geq0}\lbrace u \vert v\vert -M(u)\rbrace, 
	\end{align*}
is called an $N$-function complementary to $M (\cdot)$.
\end{definition}
The function $N(\cdot)$ can be characterized as follows. Let the function $p(\cdot): \mathbb{R}_{+} \to \mathbb{R}_{+} =[0; +\infty)$ be right continuous for $t\geq 0$, positive for $t>0$, nondecreasing and satisfy the conditions $p(0) = 0,\quad p(\infty) = \lim_{t\to\infty}p(t)=\infty$. Define
\begin{align*}
	q(s)=\sup_{p(t)\leq s} t,\quad s\geq 0.
\end{align*}
The function $q(\cdot)$ has the same properties as the function $p(\cdot)$: it is positive
for $s > 0$ , right continuous for $s \geq 0$, non-decreasing and satisfies conditions
$q (0) = 0$, $q (\infty)$ = $\lim_{s\to\infty} q (s) = \infty$. The functions
\begin{align*}
	M(u)=\int_{0}^{\vert u\vert} p(t)dt, \quad N(v)=\int_{0}^{\vert v\vert} q(s)ds,
\end{align*}
are called $N$-functions complementary to each other.

Now let’s define the Orlicz space. Consider a function $M(\cdot)$ called an $N$-function, and let $\Omega\subset \mathbb{R}^{n}$ be a measurable set in finite-dimensional space. Denote by $L_{0}(\Omega)$ the set of all functions measurable in $\Omega$. Let
\begin{align*}
	\rho_{M}(u)=\int_{\Omega} M(u(x))dx,
\end{align*}
and
\begin{align*}
	L_{M}(\Omega)=\lbrace u\in L_{0}(\Omega): \rho_{M}(u)<\infty \rbrace. 
\end{align*}
$L_{M}(\Omega)$ is called an Orlicz class.

Let $M (\cdot)$ and $N(\cdot)$ be $N$-functions complementary to each other. Let
\begin{align*}
	L^{*}_{M}(\Omega)=\lbrace u \in L_{0}(\Omega) : \vert (u,v)\vert < +\infty,\quad \forall u (\cdot) \in L^{*}_{M}(\Omega)\rbrace,
\end{align*}
where
\begin{align*}
	(u,v)=\int_{\Omega} u(x)\overline{v(x)}dx
\end{align*}
$L^{*}_{M}(\Omega)$ is called an Orlicz space. With the norm:
\begin{align*}
	\Vert u\Vert_{M}=\sup_{\rho(v,N)\leq1}\vert (u,v)\vert,
\end{align*}
$L^{*}_{M}(\Omega)$ becomes a Banach space.

Let $\Omega\in \mathbb{R}^{n}$. We will denote the characteristic form of $\Omega$ with $\chi_{\Omega}(\cdot)$, so that the norm of characteristic form  is given the following form:
\begin{align*}
	\Vert \chi_{\Omega}\Vert_{L_{M}(\mathbb{R}^{n})}=mes(\Omega) N^{-1}\bigg(\frac{1}{mes(\Omega)}\bigg).
\end{align*}
(see more information \cite{14},\cite{21},\cite{22},\cite{30})

Now we will denote the definition of $\Delta_{2}$-condition, which it is an essential for Orlicz spaces.
\begin{definition}(\cite{30})
	$N$-function $M (\cdot)$ satisfies $\Delta_{2}$-condition for large values of $u$, if
	$\exists k > 0 \quad and \exists u_{0} \geq 0$:
	\begin{align*}
		M(2u) \leq k M (u),\quad  \forall u \geq u_{0}.
	\end{align*}
	$\Delta_{2}$-condition is equivalent to requiring that, for $\forall l > 1, \exists k (l) > 0 \quad and\quad  \exists u_{0} \geq 0$:
	\begin{align*}
			M (lu) \leq  k(l) M(u),\quad  \forall u \geq u_{0}.
	\end{align*}
\end{definition}
Let’s recall the following well known fact
\begin{statement}(\cite{30}) If N-function $M (\cdot)$ satisfies $\Delta_{2}$-condition, then $L^{*}_{M}(\Omega)=L_{M}(\Omega)$ and the closure of the set of bounded (including continuous) functions coincides with $L^{*}_{M}(\Omega)$.
	\end{statement}
\begin{statement}(\cite{30}). If N-function $M (\cdot)$ satisfies $\Delta_{2}$-condition, then $L^{*}_{M}(\Omega)$ is separable.
\end{statement}
First of all, we will give some information Boyd index for Orlicz spaces , which they play very strict role in our article.

So, let $M(\cdot)$ be some N-function and $M^{-1}(\cdot)$ be its inverse on $\mathbb{R}_{+}$. Let
\begin{align*}
	h (t) =\limsup_{x\to \infty}\frac{M^{-1}(x)}{M^{-1}(tx)},\quad t>0,
\end{align*}
and define the numbers
\begin{align*}
	\alpha_{M}=-\lim_{t\to\infty}\frac{\log h(t)}{\log t}; \quad \beta_{M}=-\lim_{t\to 0+}\frac{\log h(t)}{\log t}.
\end{align*}
The numbers $\alpha_{M}$ and $\beta_{M}$ are called upper and lower Boyd indices for the Orlicz space $L_{M}$.
The following relationship holds:
\begin{align*}
	0 \leq \alpha_{M} \leq \beta_{M} \leq 1; \quad \alpha_{M} + \beta_{N} \equiv 1;\quad  \alpha_{N}+ \beta_{M} \equiv 1.
\end{align*}

The space $L_{M}$ is reflexive if and only if $0 < \alpha_{M} \leq \beta_{M} < 1$. If $1 \leq p < \frac{1}{\beta_{M}}\leq\frac{1}{\alpha_{M}}< q \leq\infty$, then the continuous embeddings $L_{q} (\Omega) \subset L_{M}(\Omega) \subset L_{p} (\Omega)$ hold.
We will also need the following
\begin{theorem}(\cite{46})
	For every p and q such that
	\begin{align*}
		1 \leq p < \frac{1}{\beta_{M}}\leq\frac{1}{\alpha_{M}}< q \leq\infty,
	\end{align*}
	we have
	\begin{align*}
		L_{q} \subset L_{M}\subset L_{p},
	\end{align*}
	with the inclusion maps being continuous.
\end{theorem}
The following interesting fact is true (see, e.g., \cite{14},\cite{32}, \cite{43}-\cite{46}).
\begin{theorem}(\cite{32})
	Let  $1<p<q<+\infty$. If the linear operator $T \in [L_{q}]$ and $T \in [L_{p}]$, then $T \in [L_{M}]$ for arbitrary Orlicz space $L_{M}$ with Boyd indices $\alpha_{M},\beta_{M}$: $\frac{1}{q}<\alpha_{M}\leq\beta_{M}<\frac{1}{p}$.
\end{theorem}
\begin{corollary}(\cite{32})
	Let $L_{M}$ be Orlicz spaces with Boyd indicies $\alpha_{M}$ and $\beta_{M}$ such that $0<\alpha_{M}\leq \beta_{M}<1$. Then the singular operator $T$ is bounded in $L_{M}$; i.e., $T \in [L_{M}]$.
\end{corollary}
\begin{theorem}(\cite{43})
	 The following statements are equivalent:
	 
	(a) The Orlicz space $L_M(\Omega)$ is reflexive.
	
	(b) The complementary Young functions M and N satisfy the $\Delta_{2}$ -condition.
	
	(c) The Boyd indices satisfy the inequalities 
	\begin{align*}
		0<\alpha_{M}\leq\beta_{M}<1.
	\end{align*}
\end{theorem}
Reflexivity of $L_{M}$ is equivalent to the condition $ 1<\alpha_{M}\leq\beta_{M}<1$. It is absolutely clear that the inclusion $L_{M}\subset L_{1}$ holds and the relation
\begin{align}
	\Vert f\Vert_{1}\leq \mathcal{C}\Vert f\Vert_{M},\quad \forall f\in L_{M},
\end{align}
is true, where $\mathcal{C}>0$ is an absolute constant.
\subsection{\textbf{Some important results}}
We will establish some important findings regarding convolutions. Let $\Omega$ be a bounded domain in $\mathbb{R}^n$ with a diameter of $d_{\Omega}$. We define $Q_d$ as the cube with sides parallel to the coordinate axes, with a length of $d_{\Omega}$ and containing $\Omega$. We assume that the center of $Q_d$ is at the origin.
We have an arbitrary function $f$ defined on $\Omega$. We extend $f$ by adding zeroes to cover the entire $Q_d$, and then create a periodic extension of $f$ with a period of $d_{\Omega}$ for all variables $x_k$, where $k=1,\dots,n$. This extended function is denoted as $f_d$.
Additionally, we have a reflexive Orlicz space $L_M$ defined on $\Omega$.

Let $k = (k_{1};\dots;k_{n}) \in \mathbb{Z}_{n}$ be a vector with integer entries. The Banach function space $L_{M}(Q_{d})$ is defined with the norm $\Vert f_{d}(\cdot+k)\Vert_{L_{M}(Q_{d})} = \Vert f_{d}(\cdot)\Vert_{L_{M}(Q_{d})} =\Vert f\Vert_{M}$. This means that the norm of $f_{d}$ is shift-invariant.

The equimeasurability of $f_{d}(\cdot)$ and $f_{d}(\cdot+y)$ on $\Omega$ for every $y\in \mathbb{R}^{n}$, suggests that the reflexive Orlicz space $L_{M}$ implies that $\Vert f_{d}(\cdot+y)\Vert_{M} = \Vert f_{d}\Vert_{M} = \Vert f\Vert_{M}$. In other words, the norm of $f_{d}$ remains the same even when shifted by $y$.

Consider the convolution of $f \in L_{M}(\Omega)$ and $g \in L_{N}(\Omega)$; i.e.,
\begin{align*}
	(f * g)(x)=&\int_{Q_{d}} f_{d}(x - y)g_{d}(y)dy,\quad  x \in Q_{d}.
\end{align*}

Since $ L_{M}(\Omega),L_{N}(\Omega)\subset L_{1}(\Omega)$, the existence of convolution $(f * g)(x)$ for a.e. $x \in Q_{d}$ is undoubted. Applying
Hölder’s inequality to this integral, we obtain
\begin{align*}
	\vert (f * g)(x) \vert\leq\Vert f_{d}(x-\cdot)\Vert_{M} \Vert g_{d}(\cdot)\Vert_{N},\quad for \quad almost \quad all\quad  x\in Q_{d}.
\end{align*}
Hence, we get
\begin{align}
	\Vert f*g\Vert_{\infty}\leq\Vert f\Vert_{M}\Vert g\Vert_{N}
\end{align}

Denote by $G_{M}(\Omega)$ the subspace of functions from $L_{M}(\Omega)$ whose shifts are continuous in $L_{M}(\Omega)$; i.e., such that it defined the following form:

\begin{align*}
	G_{M}(\Omega)=\lbrace f\in L_{M}(\Omega): \Vert T_{\delta}f-f\Vert_{M}\to 0,\quad \delta\to 0\rbrace
\end{align*}
with the norm $\Vert\cdot\Vert_{M},\quad G_{M}(\Omega)$ becomes a Banach space, where $\delta\in \mathbb{R}^{n}$ is the shift vector and $T_{\delta}f(x)=f(x+\delta)$ is the shift operator. We have
\begin{align*}
	T_{\delta}(f * g)(x)=(f * g)(x+\delta)=\int_{Q_{d}}	f(x +\delta -y)g(y)dy =(T_{\delta} f *g)(x).
\end{align*}

By the periodicity of $f$ and $g$ (with period $d_{\Omega}$), we obtain
\begin{align*}
	&T_{\delta}(f * g)(x)=\int_{Q_{d}}	f(x +\delta -y)g(y)dy
	=\int_{Q_{d}\setminus d_{\Omega}}f(x -\tau)g(\delta+\tau)d\tau\\ 
	=&\int_{Q_{d}}f(x -\tau)T_{\delta}g(\tau)d\tau=(f * T_{\delta}g)(x).
\end{align*}
Inequality (2.7) implies that
\begin{align*}
	\Vert T_{\delta}(f * g)(x)-(f * g)(x)\Vert_{\infty}=\Vert (T_{\delta}f - f)(x) * g(x)\Vert_{\infty}\leq \Vert (T_{\delta}f - f)(x)\Vert_{M}\Vert g\Vert_{N}.
\end{align*}
To get our main results, we also need the following Minkowski inequality for convolution in $Q_{d}$.
\begin{theorem}
Let $L_{M}$ be reflexive Orlicz space on $Q_{d}$. Then for every $f,g\in L_{M}$, the convolution $f*g$ belong to $L_{M}$ and 	
	\begin{align*}
		\Vert f*g\Vert_{L_{M}}\leq \Vert f\Vert_{L_{M}}\Vert g\Vert_{L_{1}(\Omega)},\quad f,g\in L_{M}.
	\end{align*}
\end{theorem}
\begin{proof}
	We have,
	\begin{align*}
			&\Vert f*g\Vert_{L_{M}}=\sup_{\rho(v,N)\leq1}\bigg\vert \int_{\Omega}(f*g)(x)v(x)dx\bigg\vert=\sup_{\rho(v,N)\leq1}\bigg\vert \int_{Q_{d}}\int_{Q_{d}}f_{d}(x-y)g_{d}(y)v_{d}(x)dydx\bigg\vert\\
			=&\sup_{\rho(v,N)\leq1}\bigg\vert \int_{Q_{d}}\int_{Q_{d}}f_{d}(x-y)v_{d}(x)dx g_{d}(y)dy\bigg\vert\leq  \int_{Q_{d}}\sup_{\rho(v,N)\leq1}\bigg\vert\int_{Q_{d}}f_{d}(x-y)v_{d}(x)dx \bigg\vert \vert g_{d}(y)\vert dy\\
			=&\int_{Q_{d}} \Vert f_{d}(\cdot-y)\Vert_{M}\vert g_{d}(y)\vert dy=\Vert f\Vert_{M}\int_{Q_{d}} \vert g_{d}(y)\vert dy=\Vert f\Vert_{M}\int_{\Omega} \vert g(y)\vert dy=\Vert f\Vert_{L_{M}}\Vert g\Vert_{L_{1}(\Omega)}.
	\end{align*}
\end{proof}
From (2.6) and theorem 2.5, we will get the following results.

\begin{corollary}
	Let $L_{M}$ be reflexive Orlicz space on $\Omega$. Then for every $f,g\in L_{M}(\Omega)$,the convolution $f*g$ also belongs to $L_{M}(\Omega)$ amd $\Vert f*g\Vert_{M}\leq \mathcal{C}\Vert f\Vert_{M}\Vert g\Vert_{M}$, where $\mathcal{C}$ is a constant independent of $f$ and $g$.
\end{corollary}
By analogy to Theorem 2.5, we can prove

\begin{lemma}
	Assume $f:\Omega \times \Omega \to R$ be a (Lebesgue) measurable function, $f(\cdot,y)\in L_{M}$ for a.e, $y\in \Omega$ and $\Vert f(\cdot,y) \Vert_{L_{M}} \in L_{1}(\Omega)$. Then 
		\begin{align*}
			\bigg\Vert \int_{\Omega} f(\cdot,y)dy\bigg\Vert_{L_{M}}\leq \int_{\Omega}\Vert f(\cdot,y)\Vert_{L_{M}}dy
		\end{align*}
	Theorem 2.5 yields
	\end{lemma}
 The following lemma is true.
	\begin{lemma} Let $L_{M}$ be reflexive Orlicz space on $\Omega$ such that $\Vert \chi_{E}\Vert_{L_{M}}\to 0$ as $\vert E\vert \to 0$. Hence
		$\overline{C^{\infty}_{0}(\Omega)}$=$G_{M}(\Omega)$ .
	\end{lemma}
\begin{proof}
	Assume $\omega_{\varepsilon}(\cdot)$ be an $\varepsilon$-cap, i.e, for $x\in Q_{d}$. 
\begin{align*}
	\omega_{\varepsilon}(x)=\begin{cases}
		C_{\varepsilon}e^{-\frac{\varepsilon^{2}}{\varepsilon^{2}-\vert x\vert^{2}}}, \quad \vert x\vert<\varepsilon,\\
		0,\quad\quad\quad\quad\quad \vert x\vert>0.
	\end{cases}
\end{align*}
where $C_{\varepsilon}$ is a constant such that
\begin{align*}
	\int_{\mathbb{R}^{n}}\omega_{\varepsilon}(x)dx=1, 
\end{align*}

Extend $\omega_{\varepsilon}(\cdot)$ to the whole of $\mathbb{R}^{n}$ periodically with period $d_{\Omega}$. Assume $f\in G_{M}$ be an any function. Mention the convolution between $f$ and $\omega_{\varepsilon}$ by $f_{\varepsilon}$, i.e. $f_{\varepsilon}=f*\omega_{\varepsilon}$.

\begin{align*}
	&f_{\varepsilon}(x)=\int_{Q_{d}}f_{d}(y)\omega_{\varepsilon}(x-y)dy=\int_{Q_{d}}f_{d}(x-y)\omega_{\varepsilon}(y)dy=(f*\omega_{\varepsilon})(x)
\end{align*}
we obtain
\begin{align*}
	&\Vert f-f_{\varepsilon}\Vert_{M}=\Bigg\Vert \int_{Q_{d}}f(\cdot)\omega_{\varepsilon}(y)dy-\int_{Q_{d}}f(\cdot-y)\omega_{\varepsilon}(y)dy\bigg\Vert_{M}\\
	=&\bigg\Vert\int_{Q_{d}}(f(x)-f(x-y))\omega_{\varepsilon}(y)dy\bigg\Vert_{L_{M}}
	\leq \int_{Q_{d}}\Vert f(\cdot)-f(\cdot-y)\Vert_{L_{M}}\omega_{\varepsilon}(y)dy\\
	\leq& \sup_{\vert y\vert\leq\varepsilon}\Vert f(\cdot-y)-f(\cdot)\Vert_{M}\to 0. \quad \varepsilon\to 0.
\end{align*}
As $f_{\varepsilon}\in C^{\infty}(\bar{\Omega})$ ($C^{\infty}(\bar{\Omega})$ is a space of infinitely differentiable functions on $\bar{\Omega})$, it directly follows that $\overline{C^{\infty}(\bar{\Omega})}=G_{M}(\Omega)$. Consequently, for every $\varepsilon>0$, there exists $g\in C^{\infty}(\bar{\Omega})$ we get
\begin{align*}
	 \Vert f-g\Vert_{M}<\varepsilon.
\end{align*}
Assume that $\delta \in \mathbb{R}^{n}$ be so that $\vert \delta\vert>0$ is some enough small number. Suppose that
\begin{align*}
	\Omega_{\delta}=\lbrace x\in \Omega: \quad \rho(x;\partial\Omega)\geq \vert \delta\vert\rbrace.
\end{align*}
It is obvious that $mes(\Omega\setminus \Omega_{\delta})\to 0, \quad \delta\to 0$. Assume
\begin{align*}
	g_{\delta}(x)=\begin{cases}
		g(x), \quad x\in \Omega_{\delta},\\
		0,\quad \quad  x \not\in \Omega_{\delta}.
	\end{cases}
\end{align*}
Also suppose $g_{\delta, \varepsilon}(x)=(g_{\delta}*\omega_{\varepsilon})(x)$. It is not difficult to see that $g_{\delta, \varepsilon}(x)\in C^{\infty}_{0}(\Omega)$ for $\varepsilon<\frac{\delta}{2}$. We obtain
\begin{align*}
	&\Vert g-g_{\delta}\Vert_{M}=\sup_{\rho(v,N)\leq1}\bigg\vert \int_{\Omega\setminus \Omega_{\delta}} g(x)v(x)dx\bigg\vert
	\leq\sup_{\rho(v,N)\leq1}  \int_{\Omega\setminus \Omega_{\delta}} \vert g(x)\vert\vert v(x)\vert dx\\
	\leq &\Vert g\Vert_{\infty} \sup_{\rho(v,N)\leq1}  \int_{\Omega\setminus \Omega_{\delta}}\vert v(x)\vert dx\leq \Vert g\Vert_{\infty} mes\big(\Omega\setminus \Omega_{\delta}\big)N^{-1}\bigg(\frac{1}{mes(\Omega\setminus \Omega_{\delta})}\bigg)\to 0, \quad \delta\to 0.\\
\end{align*}
On the other hand, $\big\Vert g_{\delta}-g_{\delta, \varepsilon}\big\Vert_{M}\to 0, \varepsilon\to 0$. Hence, from the estimate
\begin{align*}
	&\Vert f-g_{\delta, \varepsilon}\Vert_{M}\leq \Vert f-g\Vert_{M}+\Vert g-g_{\delta}\Vert_{M}+\Vert  g_{\delta}-g_{\delta, \varepsilon}\Vert_{M}\\
&	\implies \Vert f-g_{\delta, \varepsilon}\Vert_{M}\to 0,\quad 0<\varepsilon<\frac{\vert \delta\vert}{2}\to 0, \quad \delta \to 0.
\end{align*}
It directly follows $\overline{C^{\infty}_{0}(\Omega)} =G_{M}(\Omega)$.
\end{proof}

Consider the following singular kernel
\begin{align*}
	k(x)=\frac{\omega(x)}{\vert x\vert^{n}},
\end{align*}
where $\omega(x)$ is a positive homogeneous function of degree zero, which is infinitely differentiable and satisfies the following condition:
\begin{align*}
	\int_{\vert x\vert=1}\omega(x)d\sigma=0,
\end{align*}
$d\sigma$ being a surface element on the unit sphere. Denote by K the corresponding singular integral 
\begin{align*}
	(Kf)(x)=k*f(x)=\int_{\Omega}f(y)k(x-y)dy.
\end{align*}

Let us represent that $G_{M}(\Omega)$ is an invariant subspace of the operator K. To do so, it is enough to represent that 
\begin{align*}
	\Vert T_{\delta}Kf-Kf\Vert_{M}\to0, \quad \delta\to 0.
\end{align*}
Assume $\varepsilon>0$ be any given number. Hence there exists $g\in C^{\infty}_{0}(\Omega)$: $\Vert f-g\Vert<\varepsilon$. We get
\begin{align}
	&\Vert (Kf)(x+\delta)-(Kf)(x)\Vert_{M}=\Vert (K(f-g))(x+\delta)+(Kg)(x+\delta)-(K(f-g))(x)-(Kg)(x)\Vert_{M}\nonumber\\
	\leq& \Vert (K(f-g))(\cdot+\delta)\Vert_{M}+\Vert K(f-g)(\cdot)\Vert_{M}+\Vert (Kg)(\cdot+\delta)-(Kg)(\cdot)\Vert_{M}
\end{align}
Assume 
\begin{align*}
	Q_{d}(\delta)=\lbrace y-\delta : y\in Q_{d}\rbrace=\lbrace\tau: t+\delta\in Q_{d}\rbrace.
\end{align*}
Therefore
\begin{align*}
	(K(f-g))(x+\delta)=&\int_{Q_{d}}k_{d}(x+\delta-y)(f_{d}-g_{d})(y)dy\\
	=& \int_{Q_{d}(\delta)}k_{d}(x-\tau)(f_{d}-g_{d})(\tau+\delta)d\tau\\
	&(by \quad the \quad periodicity\quad  of\quad the\quad integrands)&\\
	=& \int_{Q_{d}}k_{d}(x-\tau)(f_{d}-g_{d})(\tau+\delta)d\tau.
\end{align*}
Hence
\begin{align*}
	\Vert (K(f-g))(\cdot+\delta)\Vert_{M}\leq \Vert K\Vert_{[L_{M}]}\Vert (f_{d}-g_{d})(\cdot+\delta)\Vert_{M},
\end{align*}
It is clear that
\begin{align*}
	\Vert (f_{d}-g_{d})(\cdot+\delta)\Vert_{M}=\Vert (f_{d}-g_{d})(\cdot)\Vert_{M}=\Vert f-g\Vert_{M}\leq\varepsilon
\end{align*}
From (2.7) we will get
\begin{align*}
		&\Vert (Kf)(x+\delta)-(Kf)(x)\Vert_{M}\leq \Vert (K(f-g))(\cdot+\delta)\Vert_{M}+\Vert K(f-g)(\cdot)\Vert_{M}+\Vert (Kg)(\cdot+\delta)-(Kg)(\cdot)\Vert_{M}\\
		\leq&\Vert K\Vert_{[L_{M}]}\Vert (f-g)(\cdot+\delta)\Vert_{M}+\Vert K\Vert_{[L_{M}]}\Vert f-g\Vert_{L_{M}(\Omega)}+\Vert (Kg)(\cdot+\delta)-(Kg)(\cdot)\Vert_{M}\\
		\leq& 2\Vert K\Vert_{[L_{M}]} \varepsilon+\Vert (Kg)(\cdot+\delta)-(Kg)(\cdot)\Vert_{M}
\end{align*}
It is suffices to show that
\begin{align*}
	\Vert (Kg)(\cdot+\delta)-(Kg)(\cdot)\Vert_{M}\to 0, \quad \delta\to 0, \quad for \quad g\in C^{\infty}_{0}(\Omega).
\end{align*}
\begin{align*}
	(Kg)(x+\delta)=\int_{Q_{d}} k_{d}(x-\tau)g_{d}(\tau+\delta)d\tau,
\end{align*}
\begin{align}
	\Vert (Kg)(\cdot+\delta)-(Kg)(\cdot)\Vert_{M}=&\bigg\Vert \int_{Q_{d}}[g_{d}(y+\delta)-g_{d}(y)]k_{d}(x-y)dy\bigg\Vert_{M}\nonumber\\
	\leq& \Vert K\Vert_{[L_{M}]} \Vert g_{d}(\cdot+\delta)-g_{d}(\cdot)\Vert_{M},
\end{align}
The uniform continuity of $g_{d}$ in $\mathbb{R}^{n}$ implies the existence of $\delta_{0} > 0$ such that $\vert g_{d}(x + \delta) - g_{d}(x)\vert < \varepsilon $ for
any $x \in \mathbb{R}^{n}$ and $\delta \in \mathbb{R}^{n}$ such that $\vert \delta\vert < \delta_{0}$ ; therefore, $\Vert g_{d}(\cdot + \delta) -g_{d}(\cdot)\Vert_{L^{\infty}(\Omega_{d})} <\varepsilon$. By using this fact in (2.8), we achieve
\begin{align*}
		\Vert (Kg)(\cdot+\delta)-(Kg)(\cdot)\Vert_{M}=\Vert K\Vert_{[L_{M}]}\Vert 1\Vert_{M}
\end{align*}
for all $\delta \in \mathbb{R}^{n}$ such that $\vert\delta\vert <\delta_{0}$. Hence, $Kg \in G_{M}$.
It is easy to note that if $\alpha_{M} > 0$, then $\Vert \chi_{E}\Vert \to 0$ as$\vert E\vert\to0$. Indeed, it is clear that in this case there exists $p \in (1,+\infty)$ such that $1/\alpha_{M}< p$. By Theorem 2.2, the continuous embedding $L_{p} \subset L_{M}$ is valid, i.e., there exists $\mathcal{C}>0$ such that $\Vert f\Vert_{L_{M}(\Omega)} \leq \mathcal{C}\Vert f\Vert_{L^{p}(\Omega)}$. Thus, $\Vert \chi_{E}\Vert_{M}\leq \mathcal{C} \Vert \chi_{E}\Vert_{L^{p}(\Omega)}=\mathcal{C}\vert E\vert^{1/p}\to 0$, as $\vert E\vert\to 0$.
As a result, the following lemma is proved.
\begin{lemma}Let $L_{M}$ be Orlicz space on $\Omega$ with  Boyd indices $\alpha_{M},\beta_{M} \in (0,1)$.
	$G_{M}(\Omega)$, is an invariant subspace of the singular operator $K$ in $L_{M}(\Omega)$.
\end{lemma}

	\section{\textbf{\textit{The Orlicz-Sobolev space $W^{m}_{G_{M},d_{\Omega}}(\Omega)$. Main lemma}}}	
	In this section, we will begin by introducing the Orlicz-Sobolev space in the following manner.
	We will denote the Orlicz-Sobolev space as $W^{m}_{M}(\Omega)$ which is defined using the following norm:
	
	\begin{align*}
			\Vert f\Vert_{W^{m}_{M}(\Omega)}=\sum_{\vert p\vert=0}^{m}\Vert \partial ^{p}f\Vert_{L_{M}},
	\end{align*}
	
	Here, $\partial ^{p}f$ represents the generalized derivative of a function $f\in L_{M}(\Omega)$ according to Sobolev's definition.

	Assume
	\begin{align*}
		W^{m}_{G_{M}}(\Omega)=\lbrace f\in W^{m}_{M}(\Omega): \Vert T_{\delta}f-f\Vert_{W^{m}_{M}}\to 0, \quad \delta\to 0\rbrace.
	\end{align*}

In the sequel, when $\Omega=\mathbb{B}{r}$, these spaces are denoted by $L_{M}(r)$, $G_{M}(r)$, $W^{m}_{M}(r)$, and $W^{m}_{G_{M}}(r)$. Along with $W^{m}_{G_{M}}(\Omega)$, we consider the space of functions $W^{m}_{G_{M}.d_{\Omega}}(\Omega)$, equipped with the norm
\begin{align*}
		\Vert f\Vert_{W^{m}_{G_{M},d_{\Omega}}(\Omega)}=\sum_{\vert p\vert \leq m} d^{\vert p\vert}_{\Omega} \Vert \partial ^{p}f\Vert_{L_{M}(\Omega)},
\end{align*}
where $d_{\Omega}=diam(\Omega)$. It can be seen that the norms of the spaces $W^{m}_{G_{M}}(\Omega)$ and $W^{m}_{M.d_{\Omega}}(\Omega)$ are equivalent to each other, and therefore their sets of functions are the same. Thus, it is sufficient to prove the existence of a solution to the equation $Lu=f$ in $W^{m}_{G_{M}}(\Omega)$ for the space $W^{m}_{G_{M}.d_{\Omega}}(\Omega)$. First, let us introduce the following definition.

	\begin{definition}
		We will state that the operator L has the property $\mathscr{P}_{x_{0}})$ if its coefficients holds the following conditions:
		
		1) $a_{p}\in L_{\infty}(\mathbb{B}_{r}(x_{0})), \forall\vert p\vert\leq m$, for some $r>0$,
		
		2) There is a positive value of $r$, such that when the absolute value of $p$ is equal to $m$, the coefficient $a_{p}(\cdot)$ is the same almost everywhere in $\mathbb{B}_{r}(x_{0})$ as a function that is both bounded and continuous at the point $x_{0}$.
	\end{definition}

The statement that L has the property $\mathscr{P}_{x{0}}$ for all $x_0 \in \Omega$ is true if the coefficients $a_{p}$ belong to the space $C(\Omega)$ for all $\vert p\vert\leq m$.

We can examine the elliptic operator of order m with coefficients $a_p$ given by equation (2.1) and the corresponding operator $\mathbb{T}_{x_{0}}$ given by equation (2.4). The operators corresponding to the point $x_{0}=0$ can be denoted by $\mathcal{S}_{0}$,  $L_{0}$, and $\mathbb{T}_{0}$. The main lemma can be proven as follows.
	\begin{lemma}(\textbf{Main Lemma})		
		Let $L_{M}$ be an Orlicz space defined on $\Omega$ with Boyd indices $\alpha_{M},\beta_{M} \in (0,1)$. Consider an elliptic operator $L$ of order $m$ that has the property $\mathscr{P}_{x_{0}})$ at a point $x_{0}\in \Omega$, and let $\varphi\in \dot{W}G^{m}_{M}(\mathbb{B}{r}(x_{0})), r>0$. Hence
		\begin{align*}
				\Vert \mathbb{T} _{0}\varphi\Vert_{W^{m}_{G_{M},d_{\Omega}}(\mathbb{B}_{r}(x_{0}))}\leq& \sigma(r)\Vert \varphi\Vert_{W^{m}_{G_{M},d_{\Omega}}(\mathbb{B}_{r}(x_{0}))},
		\end{align*}
		where the constant $\sigma(r)\to 0$  as $r\to 0$, and it depends only on the ellipticity constants $\mathrm{L}_{x_{0}}$ and the coefficients of the operator $\mathrm{L} $.
	\end{lemma}
	\begin{proof}
		We will use the idea of [2, p.236]. Without spoiling the generality we suppose $x_{0}=0$. With the value of the function $a_{p}(\cdot)$ at the point $x_{0}=0$ for $\vert p\vert=m$  we will consider the value of the corresponding function from the continuous property $P_{x_{0}})$ at this point. Due to simplicity, we suppose that $n\geq 3$ and odd, with $r<1$. Following [2], we suppose
		\begin{align*}
			\psi=&(\mathrm{L} _{0}-\mathrm{L} )\varphi=\psi_{1}+\psi_{2},\\
			\psi_{1}(x)=&\sum_{\vert p\vert =m} (a_{p}(0)-a_{p}(x))\partial ^{p}\varphi(x)=\sum_{\vert p\vert =m} b_{p}(x)\partial ^{p}\varphi(x),
		\end{align*}
		where $b_{p}(x)=a_{p}(0)-a_{p}(x)$. Clearly, $b_{p}(0)=0$, and as a result, $\sup_{\vert x\vert <r}\vert b_{p}(x)\vert=\bar{\bar{o}}(1), t\to 0$;
		\begin{align*}
			\psi_{2}(x)=&-\sum_{\vert p\vert <m} a_{p}(x)\partial ^{p}\varphi(x), \quad r\to 0.
		\end{align*}
		It is clear that
		\begin{align*}
			\Vert \psi_{1}\Vert_{L_{M}}(r)\leq& \bar{\bar{o}}(1)\sum_{\vert p\vert =m} \Vert \partial ^{p}\varphi(x)\Vert_{L_{M}}(r),\quad r\to 0,
		\end{align*}
		satisfies.
		
		Let $\chi=\mathbb{T} _{0}\varphi$. Considering the expression for $\mathbb{T} _{0}$, we get
		\begin{align*}
			\chi=\mathcal{S} _{0}(\mathrm{L} _{0}-\mathrm{L} )\varphi=\mathcal{S} _{0}\phi=\int_{\mathbb{B}_{r}} \mathbb{J}_{0}(x-y)\phi(y)dy.
		\end{align*}
		For the $\vert p\vert <m$, we obtain 
		\begin{align*}
			\partial ^{p}\chi(x)=\int_{\mathbb{B}_{r}}	\partial ^{p}_{x}\mathbb{J}_{0}(x-y)\phi(y)dy.
		\end{align*}
		In this part, the following estimate is valid for the derivatives $\partial ^{p}\mathbb{J}_{0}$:
		\begin{align*}
			\vert \partial ^{p}\mathbb{J}_{0}(x)\vert \leq \mathcal{C}\vert x\vert ^{m-n-\vert p\vert}.
		\end{align*}
		Hence for $	\partial ^{p}\chi(x)$ we get
		\begin{align}
			\vert 	\partial ^{p}\chi(x)\vert \leq \mathcal{C}\int_{\vert y\vert <r}\vert x-y\vert^{m-n-\vert p\vert}\vert \psi(y)\vert dy.
		\end{align}
		We will use the following popular formula for $\alpha<n$.
		\begin{align}
			\int_{\vert x- y\vert <r}\frac{dy}{\vert x-y\vert^{\alpha}}=\frac{\vert \mathbb{B}_{1}\vert r^{n-\alpha}}{n-\alpha}, \quad \forall x\in \mathcal{R} ^{n},
		\end{align}
		where $\vert \mathbb{B}_{1}\vert$ is a volume of a unit ball $ \mathbb{B}_{1}$ in $\mathcal{R} ^{n}$.(see, e.g. 4, p.19).
		
		\begin{align*}
			\mathscr{I}(x)=\int_{\vert y\vert <r}\vert x-y\vert^{m-n-\vert p\vert}\vert \psi(y)\vert dy.
		\end{align*}
	To this aim, define
		\begin{align*}
			f(x)=&
			\begin{cases*}
				\vert x\vert^{m-n-\vert p\vert}, \quad \vert x\vert <r,\\
				0,\quad \quad \quad\quad \quad  \vert x\vert \geq r.
			\end{cases*}	\quad\quad
			g(x)=\begin{cases*}
				\psi(y), \quad \vert x\vert <r,\\
				0,\quad \quad \vert x\vert \geq r.
			\end{cases*}	
		\end{align*}
	Hence it is not difficult to see that $supp(f*g)\subset B_{2r}$, and therefore, by theorem 2.4, we get
		\begin{align}
			\Vert \mathscr{I}(\cdot)\Vert_{M}\leq& \Vert f*g\Vert_{M}\leq \Vert f\Vert_{L_{1}(\mathcal{R} ^{n})}\Vert g\Vert_{M}.
		\end{align}
	Applying the formula (3.2) to $\Vert f\Vert_{L_{1}(\mathcal{R} ^{n})}$, we have
		\begin{align*}
			\Vert f\Vert_{L_{1}(\mathcal{R} ^{n})}=\Vert f\Vert_{L_{1}(r)}=\int_{\mathbb{B}_{r}}\vert x\vert^{m-n-\vert p\vert}dx=\frac{\vert \mathbb{B}_{1}\vert2^{m-\vert p\vert}}{m-\vert p\vert}r^{m-\vert p\vert}.
		\end{align*}
	From above expression and (3.1), we achieve
		\begin{align*}
			\Vert \partial ^{p}\chi\Vert_{L_{M}(r)}\leq \mathcal{C}\Vert \mathscr{I}(x)\Vert_{L_{M}(r)}\leq \mathcal{C}r^{m-\vert p\vert}\Vert \psi \Vert_{L_{M}(r)},
		\end{align*}
	Then, we get
		\begin{align*}
			r^{\vert p\vert}\Vert \partial ^{p}\chi\Vert_{L_{M}(r)}\leq \mathcal{C}r^{m}\Vert \psi \Vert_{L_{M}(r)}.
		\end{align*}
     Consequently, taking into account the estimate for $\Vert \psi_{1}\Vert_{L_{M}(r)}$, we get
		\begin{align*}
			&	r^{\vert p\vert}\Vert \partial ^{p}\chi\Vert_{L_{M}(r)}\leq \mathcal{C}r^{m} (\Vert \psi_{1}\Vert_{L_{M}(r)}+\Vert \psi_{2}\Vert_{L_{M}(r)} )\\
			\leq&\mathcal{C}r^{m}\bigg( \bar{\bar{o}}(1)\sum_{\vert p\vert =m} \Vert \partial ^{p}\varphi\Vert_{L_{M}(r)}+\sum_{\vert p\vert<m} \Vert \partial ^{p}\varphi\Vert_{L_{M}(r)}\bigg)\\
			\leq&\mathcal{C}\bigg(\bar{\bar{o}}(1)\sum_{\vert p\vert =m}r^{m}\Vert \partial ^{p}\varphi\Vert_{L_{M}(r)}+ \sum_{\vert p\vert<m} r^{m-\vert p\vert} r^{\vert p\vert} \Vert \partial ^{p}\varphi\Vert_{L_{M}(r)}\bigg)\\
			\leq&\mathcal{C}\bigg(\bar{\bar{o}}(1)\sum_{\vert p\vert =m}r^{m}\Vert \partial ^{p}\varphi\Vert_{L_{M}(r)}+ r \sum_{\vert p\vert<m}r^{\vert p\vert} \Vert \partial ^{p}\varphi\Vert_{L_{M}(r)}\bigg)\\
			\leq& \sigma(r) \Vert \varphi\Vert _{W^{m}_{G_{M},d_{\Omega}}(r)},
		\end{align*}
		for $r$ small enough and $\sigma(r)$ vanishing function as $r\to 0$. Consequently, the following inequality is correct
		\begin{align}
			&	r^{\vert p\vert}\Vert \partial ^{p}\chi\Vert_{L_{M}(r)}\leq\sigma(r) \Vert \varphi\Vert _{W^{m}_{G_{M},d_{\Omega}}(r)},\quad r\to 0,\quad \forall \vert p\vert<m.
		\end{align}
		Let us, we consider the case $\vert p\vert=m$. In this case we obtain
		\begin{align}
			\partial ^{p} \chi(x)=&\int_{\mathbb{B}_{r}}\partial ^{p}\mathbb{J}_{0}(x-y)\psi(y)dy+\mathcal{C}\psi(x), \quad for\quad a.e.\quad x\in \mathbb{B}_{r}
		\end{align} 
		(see [2, p. 235]). Then $\partial ^{p}\mathbb{J}_{0}(x)$ is a singular kernel. Applying theorem 2.5, from (3.5) we get
		\begin{align}
			\Vert \partial ^{p}\chi(x)\Vert_{L_{M}(r)}\leq&\mathcal{C}\Vert \psi\Vert_{L_{M}(r)},\quad \vert p\vert =m.
		\end{align}
		Then the following estimate satisfies
		\begin{align*}
			r^{\vert p\vert}\Vert \partial ^{p}\chi\Vert_{L_{M}(r)}\leq	&\mathcal{C}r^{m} \Vert \psi\Vert_{L_{M}(r)} \leq\mathcal{C}r^{m} (\Vert \psi_{1}\Vert_{L_{M}(r)}+\Vert \psi_{2}\Vert_{L_{M}(r)} )\\
			\leq&\mathcal{C}r^{m}\bigg( \bar{\bar{o}}(1)\sum_{\vert p\vert =m} \Vert \partial ^{p}\varphi\Vert_{L_{M}(r)}+\sum_{\vert p\vert<m} \Vert \partial ^{p}\varphi\Vert_{L_{M}(r)}\bigg)\\
			\leq&\mathcal{C}\bigg(\bar{\bar{o}}(1)\sum_{\vert p\vert =m}r^{m}\Vert \partial ^{p}\varphi\Vert_{L_{M}(r)}+ \sum_{\vert p\vert<m} r^{m-\vert p\vert} r^{\vert p\vert} \Vert \partial ^{p}\varphi\Vert_{L_{M}(r)}\bigg)\\
			\leq&\mathcal{C}\bigg(\bar{\bar{o}}(1)\sum_{\vert p\vert =m}r^{m}\Vert \partial ^{p}\varphi\Vert_{L_{M}(r)}+ r \sum_{\vert p\vert<m}r^{\vert p\vert} \Vert \partial ^{p}\varphi\Vert_{L_{M}(r)}\bigg)\\
			\leq& \mathcal{C}\bar{\bar{o}}(1) \Vert \varphi\Vert _{W^{m}_{G_{M},d_{\Omega}}(r)}.
		\end{align*}
		Considering (3.4) we achieve
		\begin{align*}
			\Vert \chi \Vert_{W^{m}_{G_{M},d_{\Omega}}(r)}= \Vert \mathbb{T} _{0}\varphi\Vert_{W^{m}_{G_{M},d_{\Omega}}(r)}\leq& \sigma(r)\Vert \varphi\Vert_{W^{m}_{G_{M},d_{\Omega}}(r)}, \quad \sigma(r)\to 0, \quad as \quad r\to 0.
		\end{align*}
      The lemma was proved.
	\end{proof}
	\section{Local existence theorem}In this section we establish a local existence theorem for strong solutions to the equation $Lu = \varphi$
	in $W^{m}_{G_{M}.d_{\Omega}}$. We must demonstrate that the operator $L$ acts from $W^{m}_{G_{M},d_{\Omega}}(\Omega)$ to $G_{M}(\Omega)$ for our subsequent actions to be proper. The lemma that follows is correct. 
	\begin{lemma} Let $L_{M}$ be a reflexive Orlicz space on $\Omega$ with Boyd index $\alpha_{M} > 0$ anda $a_{p}(\cdot)\in L_{\infty}(\Omega)$ for all $\vert p\vert\leq m$.  Hence the operator $L$ maps functions in $W^{m}_{G_{M}}(\Omega)$ to functions in $G_{M}(\Omega)$.
	\end{lemma}
\begin{proof}
	Proof that $f\in G_{M}(\Omega)$ and $\varphi\in L_{\infty}(\Omega)$ imply $\varphi f\in G_{M}(\Omega)$ is sufficient. There is no doubt that  $\varphi f\in L_{M}(\Omega)$. Therefore, evidence of the connection is sufficient.
	\begin{align*}
	\Vert	\mathbb{T}_{\delta}(\varphi f)-\varphi f\Vert_{M}\to 0, \quad \delta \to 0.
	\end{align*}
we get
\begin{align*}
	&\Vert \varphi(\cdot+\delta) f(\cdot+\delta)-\varphi(\cdot) f(\cdot)\Vert_{M}
	\leq \Vert \varphi(\cdot+\delta)\big( f(\cdot+\delta)-f(\cdot)\big)\Vert_{M}
	+\Vert \big(\varphi(\cdot+\delta)- \varphi(\cdot)\big) f(\cdot)\Vert_{M}=\Delta^{(1)}_{\delta}+\Delta^{(2)}_{\delta},
\end{align*}
where 
\begin{align*}
	\Delta^{(1)}_{\delta}=\Vert \varphi(\cdot+\delta)\big( f(\cdot+\delta)-f(\cdot)\big)\Vert_{M},\quad
	\Delta^{(2)}_{\delta}=\Vert \big(\varphi(\cdot+\delta)- \varphi(\cdot)\big) f(\cdot)\Vert_{M}.
\end{align*}
For $\Delta^{(1)}_{\delta}$ we obtain
\begin{align*}
	\Delta^{(1)}_{\delta}\leq \Vert \varphi\Vert_{\infty}\Vert f(\cdot+\delta)-f(\cdot)\Vert_{M}\to0,\quad \delta\to 0.
\end{align*}
In addition, let $\varepsilon>0$ be any given number. Since  $\overline{C^{\infty}_{0}(\Omega)}=G_{M}(\Omega)$ (the closure in $L_{M}(\Omega)$),  it is evident that there exists a function $ g\in C^{\infty}_{0}(\Omega)$:
\begin{align*}
	\Vert f-g\Vert_{M}<\varepsilon.
\end{align*}
Consequently, for $\Delta^{(2)}_{\delta}$ we get
\begin{align*}
	\Delta^{(2)}_{\delta}=\Vert \big(\varphi(\cdot+\delta)- \varphi(\cdot)\big) \big( f(\cdot)-g(\cdot)+g(\cdot)\big)\Vert_{M}\leq 2\Vert \varphi \Vert_{\infty} \varepsilon+ \Delta^{(3)}_{\delta},
\end{align*}
where
\begin{align*}
	 \Delta^{(3)}_{\delta}=\Vert \big(\varphi(\cdot+\delta)-\varphi(\cdot)\big)g(\cdot)\Vert_{M}.
\end{align*} 
Since $\alpha_{M} > 0$, there exists $p \in (1,+\infty)$ such that $1/\alpha_{M} <p$ ; thus,there is $\mathcal{C}>0$ for which $\Vert f\Vert_{M} \leq \mathcal{C} \Vert f\Vert_{p}, \quad f \in L_{M}(\Omega)$. Hence,
\begin{align*}
	\Delta^{(3)}_{\delta}\leq \Vert g\Vert_{\infty}\Vert \varphi(\cdot+\delta)-\varphi(\cdot)\Vert_{M}\leq \mathcal{C} \Vert g\Vert_{\infty}\Vert \varphi(\cdot+\delta)-\varphi(\cdot)\Vert_{L_{p}(\Omega)}\to 0, \quad \delta \to 0,
\end{align*}
where $\mathcal{C}>0$ is a constant independent of $\delta$.
\end{proof}

By using the Main Lemma and following (\cite{2}), we can establish the existence of a solution $u$ for the equation $Lu=f$ within the class $W^{m}_{G_{M},d_{\Omega}}(r)$. This theorem guarantees the existence of such a solution for a certain range of parameters and conditions.

 \begin{theorem} Let $L_{M}$ be  Orlicz space on $\Omega$ with Boyd indices $\alpha_{M},\beta_{M} \in (0,1)$. Suppose that $L$ be an mth order elliptic operator which has the property $\mathscr{P}_{x_{0}})$ at some point $x_{0} \in\Omega$ and $f\in G_{M}(\Omega)$. Hence, for enough small r, there is a solution of the equation $Lu=f$ belonging to the class $W^{m}_{G_{M}}(\mathbb{B}{r}(x_{0}))$.
 \end{theorem}
\begin{proof}
Let us assume that $x_{0}=0$ without loss of generality. We aim to prove that $L_{0}S_{0}\varphi=\varphi$ for all $\varphi\in G_{M}(r)$.
To clarify, consider an arbitrary function $L_{0}S_{0}\varphi=\varphi$ for all $\varphi\in G_{M}(r)$.

By using lemma 2.3, we obtain 
\begin{align*}
	\exists\lbrace \varphi_{k}\rbrace\subset C^{\infty}_{0}(r): \Vert \varphi_{k}-\varphi\Vert_{L_{M}(r)}\to 0, \quad k\to \infty.
\end{align*}

By using Theorem 2.1, $L_{0}S_{0}\varphi_{k}=\varphi_{k}, \forall k\in N$. Thus it is enough to prove that $L_{0}S_{0}$ acts boundedly in $G_{M}(r)$. We get

\begin{align*}
	L_{0}S_{0}\varphi_{k}=L_{0}\int_{\mathbb{B}_{r}}\mathbb{J}_{0}(x-y)\varphi_{k}(y)dy.
\end{align*}
Considering that $L_{0}$ is a homogeneous differential operator of mth order, by differentiation formula 3.5, for $\vert p\vert =m$ we get 

\begin{align*}
	(L_{0}S_{0}\varphi_{k})(x)=\int_{\mathbb{B}_{r}}L_{0}\mathbb{J}_{0}(x-y)\varphi_{k}(y)dy+const \varphi_{k}(x).
\end{align*}
Since $L_{0}\mathbb{J}_{0}(x)$ is a singular kernel, from  Lemma 2.4, it follows that $L_{0}S_{0}$ is bounded for the functios from $C^{\infty}_{0}(r)$ to $G_{M}(r)$ and $\overline{C^{\infty}_{0}(r)}=G_{M}(r)$ implies its boundedness in $G_{M}(r)$. Consequently, $L_{0}S_{0}=I_{G_{M}}$, where $I_{G_{M}}$ is a unit operator in $G_{M}(r)$. We achieve

\begin{align*}
	L_{0}\mathbb{T}_{0}=L_{0}S_{0}(L_{0}-L)=I_{G_{M}}(L_{0}-L)=L_{0}-L.
\end{align*}
By using above expression, the equation $Lu=f$ is able to be rewritten as follows.
\begin{align*}
	L_{0}u-L_{0}\mathbb{T}_{0}u=f\implies L_{0}(I_{W_{M}}-\mathbb{T}_{0})u=f,
\end{align*}
where $I_{W_{M}}$ is a unit operator in $W_{M}(r)$. Then we have 
\begin{align*}
	(I_{W_{M}}-\mathbb{T}_{0})u=S_{0}f.
\end{align*}
By using Main Lemma, $\big\Vert \mathbb{T}_{0}\big\Vert_{[W^{m}_{G_{M}.d_{\Omega}}]}= \bar{\bar{o}}(1), \quad r\to 0$. Thus, for enough small $r$ we  get $\big\Vert \mathbb{T}_{0}\big\Vert_{[W^{m}_{G_{M}.d_{\Omega}}]}<1$. Hence, operator $(I_{W_{M}}-\mathbb{T}_{0})$ is boundedly invertible in $W_{M}(r)$ and, by lemma 2.1, the function 
	\begin{align*}
		u=(I_{W_{M}}-\mathbb{T}_{0})^{-1} S_{0}f,
	\end{align*}
is a solution of the equation $Lu=f$.
\end{proof}
\section{Conclusion}
In conclusion, this article presents a study on a higher-order elliptic equation with nonsmooth coefficients in Orlicz spaces on a domain in $\mathbb{R}^{n}$. The article identifies a separable subspace within the Orlicz space that contains infinitely differentiable and compactly supported functions, and determines the Sobolev spaces generated by this subspace.

Furthermore, the article demonstrates the local solvability of the equation in Orlicz-Sobolev spaces, subject to specific restrictions on the coefficients of the equation and the Boyd indices of the Orlicz space. This result builds upon the previously established classical $L_{p}$ analog.

Overall, this research enhances our understanding of higher-order elliptic equations with nonsmooth coefficients in Orlicz spaces and provides valuable insights into their solvability in Orlicz-Sobolev spaces. Further investigation can be done to explore the global solvability of these equations and to generalize the results to other types of equations with nonsmooth coefficients.

\end{document}